\documentclass[amsppt,11pt]{amsart}
\usepackage{amsmath}
\usepackage{amsfonts}
\usepackage{amssymb}

\newcommand{\ep}{{\varepsilon}}

\newcommand{\intg}{{\mathbb Z}}
\newcommand{\cplx}{{\mathbb C}}
\newcommand{\ratl}{{\mathbb Q}}
\newcommand{\real}{{\mathbb R}}
\newcommand{\fifi}{{\mathbb F}_q}
\newcommand{\fifibar}{\bar{\mathbb F}_q}
\newcommand{\fifir}{{\mathbb F}_{q^r}}
\newcommand{\fifisq}{{\mathbb F}_{q^2}}

\newcommand{\Nm}{{\rm Nm}}

\newcommand{\bG}{{\bf G}}
\newcommand{\bT}{{\bf T}}
\newcommand{\bB}{{\bf B}}

\newcommand{\bN}{{\bf N}}
\newcommand{\U}{\mathrm{U}}
\newcommand{\GL}{\mathrm{GL}}

\newtheorem{thm}{Theorem}[section]
\newtheorem{lemma}{Lemma}[section]

\newtheorem{cor}{Corollary}[section]

\def\adots{\mathinner{\mkern2mu\raise0pt\hbox{.}  
\mkern2mu\raise4pt\hbox{.}\mkern1mu
\raise7pt\vbox{\kern7pt\hbox{.}}\mkern1mu}}

\numberwithin{equation}{section}

\begin{document}

\bibliographystyle{amsplain}

\title[Duality, central characters, real-valued characters]
{Alvis-Curtis duality, central characters, and real-valued characters}
\author{C. Ryan Vinroot}
  \address{Mathematics Department\\
           The College of William and Mary\\
           Williamsburg, VA  23187-8795}
   \email{vinroot@math.wm.edu}

\maketitle

\begin{abstract}
We prove that Alvis-Curtis duality preserves the Frobenius-Schur indicators of characters of connected reductive groups of Lie type with connected center.  This allows us to extend a result of D. Prasad which relates the Frobenius-Schur indicator of a regular real-valued character to its central character.  We apply these results to compute the Frobenius-Schur indicators of certain real-valued, irreducible, Frobenius-invariant Deligne-Lusztig characters, and the Frobenius-Schur indicators of real-valued regular and semisimple characters of finite unitary groups.
\end{abstract}

\section{Introduction}

Given a finite group $G$, and an irreducible finite dimensional complex representation $(\pi, V)$ of $G$ with character $\chi$, it is a natural question to ask what smallest field extension of $\ratl$ is necessary to define a matrix representation corresponding to $(\pi, V)$.  If $\ratl(\chi)$ is the smallest extension of $\ratl$ containing the values of $\chi$, then the {\em Schur index} of $\chi$ over $\ratl$ may be defined to be the smallest degree of an extension of $\ratl(\chi)$ over which $(\pi, V)$ may be defined.  One may also consider the Schur index of $\chi$ over $\real$, which is $1$ if $(\pi, V)$ may be defined over $\real(\chi)$, and $2$ if it is not.  If $\chi$ is a real-valued character, then the Schur index of $\chi$ over $\real$ indicates whether $(\pi, V)$ may be defined over the real numbers.  The Brauer-Speiser Theorem states that if $\chi$ is a real-valued character, then the Schur index of $\chi$ over $\ratl$ is either $1$ or $2$, and if the Schur index of $\chi$ over $\real$ is $2$, then the Schur index of $\chi$ over $\ratl$ is $2$.

As finite groups of Lie type are of fundamental importance in the theory of finite groups, it is of interest to understand the Schur indices over $\ratl$ of their complex representations.  In work of Gow and Ohmori \cite{Gow1, Gow2, Gow3, Ohm1, Ohm2}, the Schur indices over $\ratl$ of many irreducible characters of finite classical groups are determined.  For the characters which are not covered by methods of Gow and Ohmori, it seems that the computation of the Schur index over $\ratl$ is significantly more difficult.  One example is that of the special linear group, which was completed by Turull \cite{turull}.  Through methods developed by Lusztig, Geck, and Ohomri, the Schur index of unipotent characters have been calculated in some cases \cite{Lusztig, Geck1, Geck2, Ohm3}.  From all of the evidence which is known, it has been conjectured that the Schur index over $\ratl$ of any irreducible character of a finite group of Lie type is at most $2$.

A somewhat more managable problem is to find the Schur index over $\real$ of real-valued characters of finite groups of Lie type.  Because of the Brauer-Speiser Theorem, this is directly related to the question of the Schur index over $\ratl$.  In many cases of classical groups, the Schur index over $\real$ of real-valued characters is completely known, as in \cite{Gow4}.  The Schur index over $\real$ can also be obtained from the Schur index over $\ratl$, as in many examples in \cite{Gow1, Ohm1}.  In many of these examples, the Schur index over $\real$ of a character is related to the value of the central character at a particular element.  In \cite{prasad}, Prasad studies the relationship between the Schur index over $\real$ and the central character for connected reductive groups of Lie type.  In particular, it is proven in \cite{prasad} that in the case that the underlying algebraic group has connected center, there is always a central element whose value on a central character of a real-valued character indicates the Schur index over $\real$, when that real-valued character appears in the Gelfand-Graev character.

The organization and main results of this paper are as follows.  In Section \ref{FS}, we gather results on real-valued characters and real representations of finite groups.  In Section \ref{Duality}, the main result is Theorem \ref{FSdual}, where we prove that the Frobenius-Schur indicator (and thus the Schur index over $\real$) of a character is preserved under Alvis-Curtis duality.  In Section \ref{CentralCharacters}, we apply Theorem \ref{FSdual} to extend the result of Prasad in \cite{prasad}, in Theorem \ref{FScentralExt}.  Finally, in Section \ref{Applications}, we use Theorem \ref{FScentralExt} to compute the Frobenius-Schur indicator of certain Frobenius-invariant Deligne-Lusztig characters (Theorem \ref{DLthm}), and the Frobenius-Schur indicators of regular and semisimple characters of finite unitary groups (Theorem \ref{FSunitary}).

\section{Frobenius-Schur indicators and real representations} \label{FS}

Let $G$ be a finite group, and let $(\pi, V)$ be an irreducible complex representation of $G$, with character $\chi$.  Frobenius and Schur defined $\ep(\pi) = 1/|G| \sum_{g \in G} \chi(g^2)$, called the {\em Frobenius-Schur indicator} of $\pi$, or of $\chi$, which they proved takes only the values $1, -1$, or $0$, and takes the value $\pm 1$ if and only if $\chi$ is real-valued.  We also write $\ep(\pi) = \ep(\chi)$.  Frobenius and Schur further proved that $\ep(\pi) = 1$ exactly when $\pi$ is a real representation, that is, when we can choose a basis for $V$ such that the matrix representation corresponding to $\pi$ with respect to this basis has image in ${\rm GL}(n, \real)$.  When $\ep(\pi)=1$, we say that $\pi$ is an {\em orthogonal} representation, and when $\ep(\pi) = -1$, we say $\pi$ is a {\em symplectic} representation (and also that the character of $\pi$ is an orthogonal or a symplectic character, respectively).

Now let $\chi$ be an irreducible character of a finite group $G$, and suppose $\ep(\chi) = 0$.  Then $\chi + \bar{\chi}$ is a real-valued character of $G$, and is in fact the character of a real representation of $G$, and this real representation is irreducible as a real representation (since its only complex irreducible subrepresentations are not real).  Similarly, if $\psi$ is an irreducible representation of $G$ such that $\ep(\psi) = -1$, then $2\psi$ is the character of a real representation (by taking the direct sum of a matrix representation corresponding to $\psi$ and its conjugate), and this real representation is irreducible as a real representation.  In particular, this gives us a basis of characters of real representations, as follows.

\begin{lemma} \label{realbasislemma} Let $G$ be a finite group.  If $\chi$ is the character of a real representation of $G$, then $\chi$ can be written uniquely as a non-negative integer linear combination of characters of the form:  $\theta$ where $\theta$ is an irreducible character and $\ep(\theta) = 1$, $2\psi$ where $\psi$ is an irreducible character and $\ep(\psi) = -1$, and $\eta + \overline{\eta}$ where $\eta$ is an irreducible character and $\ep(\eta) = 0$.
\end{lemma}

We may take advantage of the parity implications in Lemma \ref{realbasislemma},  as in the following. 

\begin{lemma} \label{linearcomblemma} Suppose that $\chi$ is an irreducible complex character, and is an integer linear combination of characters of real representations.  Then $\ep(\chi)=1$.
\end{lemma}
\begin{proof} Let $\chi = \sum_i a_i \theta_i$, where $a_i \in \intg$ and each $\theta_i$ is the character of a real representation.  Writing each $\theta_i$ as a decomposition of characters of irreducible complex representations, the total number of characters $\psi$ such that $\ep(\psi) = 0$ or $\ep(\psi)=-1$ appearing is even, from Lemma \ref{realbasislemma}.  Since $\chi$ is irreducible, we must have $\ep(\chi) = 1$.
\end{proof}

Finally, we note the following useful fact, which follows from the defintion of restriction and induction.  If $\pi$ is a representation of $G$ with character $\chi$, and $\rho$ is a representation of a subgroup $H$ of $G$ with character $\psi$, we write $\pi|_H$ for the restriction of $\pi$ to $H$, with character $\chi|_H$, and $\rho^G$ for the induced representation of $\rho$ to $G$, with character $\psi^G$.

\begin{lemma} \label{inducerestrictlemma} Let $G$ be a finite group and $H$ a subgroup of $G$.  If $\pi$ is a real representation of $G$, then the restriction $\pi|_H$ is a real representation of $H$.  If $\rho$ is a real representation of $H$, then the induced representation $\rho^G$ is a real representation of $G$.
\end{lemma}

\section{Alvis-Curtis duality and the Frobenius-Schur indicator}  \label{Duality}

In this section, we first let $G$ be a finite group, $P$ a subgroup of $G$, and $N$ a normal subgroup of $P$.  Given a character $\chi$ of $G$, we define a character of $P$ by {\it truncation} with respect to $N$, defined in \cite{alvis1,curt,Ka81}.  In particular, the character $T_{P}(\chi)$, or $T_{P/N}(\chi)$, is defined to be the character of $P$ obtained by first restricting $\chi$ to $P$, and then taking the sum of irreducible constituents of $\chi|_P$ which are characters of representations such that $N$ acts trivially.  By extending the definition linearly, we have that $T_{P/N}$ maps class functions of $G$ to class functions of $P$.  Note that since $N$ is acting trivially, we may also view $T_{P/N}(\chi)$ as a class function on $P/N$.  In terms of character values, we have the formula
\begin{equation} \label{truncform}
T_{P/N}(\chi)(h) = \frac{1}{|N|} \sum_{x \in N} \chi(xh), \quad h \in P.
\end{equation}
As the next result states, truncation preserves the property of being the character of a real representation. 

\begin{lemma} \label{truncreal} If $\chi$ is the character of a real representation of $G$, then $T_{P/N}(\chi)$ is the character of a real representation of $P$.\end{lemma}
\begin{proof} First, the restriction $\chi|_P$ is the character of a real representation of $P$, by Lemma \ref{inducerestrictlemma}.  Now consider the irreducible constituents of $\chi|_P$ which are characters of representations such that $N$ acts trivially.  If $\xi$ is the character of such a representation, and $\ep(\xi) = 0$, then certainly $\bar{\xi}$ is also such a character.  Because $\chi|_P$ is the character of a real representation, the irreducible characters $\psi$ of $P$ such that $\ep(\psi) = -1$ must occur with even multiplicity as constituents of $\chi|_P$, by Lemma \ref{realbasislemma}.  Thus, such characters must appear with even multiplicity in $T_{P/N}(\chi)$.  It follows from Lemma \ref{realbasislemma} that $T_{P/N}(\chi)$ is the character of a real representation of $P$.   
\end{proof} 

Let $\fifi$ be a finite field of $q$ elements, where $q$ is a power of the prime $p$, and let $\bar{\mathbb{F}}_q$ be a fixed algebraic closure of $\fifi$.  Now let ${\bf G}$ be a connected reductive group over $\bar{\mathbb{F}}_q$ which is defined over $\fifi$ and which has connected center, and let $F$ be a Frobenius map.  Let $W$ be the Weyl group of ${\bf G}$, where $W = \langle s_i \mid i \in I \rangle$, and let $\rho$ be the permutation of the indexing set $I$ which is induced by the action of the Frobenius map $F$.  For any $\rho$-stable subset $J \subseteq I$, let ${\bf P}_J$ be the parabolic subgroup of ${\bf G}$ corresponding to $W_J = \langle s_j \mid j \in J \rangle$, and let ${\bf N}_J$ be the unipotent subgroup.  Let $P_J = {\bf P}_J^F$ and $N_J = {\bf N}_J^F$ be the corresponding parabolic and unipotent subgroups of the finite group $G = {\bf G}^F = {\bf G}(\fifi)$.  Define the following operator $*$ on the set of virtual characters of $G$:
$$\chi^* = \sum_{ J \subseteq I \atop {\rho(J) = J}} (-1)^{|J/\rho|} (T_{P_J/N_J}(\chi))^G.$$
The operator $*$ and its properties were studied in \cite{alvis1, alvis2, alvis3, curt, Ka81}.  In particular, if $\chi$ is an irreducible character of $G$, then $\pm \chi^*$ is an irreducible character of $G$.  By a slight abuse of notation, we let $\pm \chi^*$ denote the appropriate sign taken to get an irreducible character.  Another key important property of $*$ is as follows, which is why $*$ is called the {\em Alvis-Curtis duality} operator.

\begin{thm} [Curtis, Alvis, Kawanaka]  
The map $\chi \mapsto \chi^*$ is an order 2 isometry of the virtual characters of $G$, so that $\chi^{**} = \chi$ and $\langle \chi, \psi \rangle = \langle \chi^*, \psi^* \rangle$ for all virtual characters $\chi, \psi$ of $G$.
\end{thm}

It follows from the definition of $*$ that it preserves real-valued characters, since $(\overline{\chi})^* = \overline{\chi^*}$.  We now prove the result that duality in fact preserves the property of being the character of a real representation, or preserves the Frobenius-Schur indicator.

\begin{thm} \label{FSdual} Let $\chi$ be an irreducible character of $G = {\bf G}(\fifi)$.  Then $\ep(\chi) = \ep(\pm\chi^*)$.
\end{thm}
\begin{proof} Since $(\overline{\chi})^* = \overline{\chi^*}$, and $\chi^{**} = \chi$, it is enough to show that if $\ep(\chi) = 1$, then $\ep(\pm\chi^*) = 1$.  So, suppose that $\chi$ is an irreducible character of $G$ corresponding to a real representation of $G$.  From Lemma \ref{truncreal}, it follows that $T_{P_J/N_J}(\chi)$ is the character of a real representation of $P_J$ for each $P_J$, and so each induced representation $(T_{P_J/N_J}(\chi))^G$ is the character of a real representation of $G$, by Lemma \ref{inducerestrictlemma}.  Now, in the sum
$$\chi^* = \sum_{ J \subseteq I \atop {\rho(J) = J}} (-1)^{|J/\rho|} (T_{P_J/N_J}(\chi))^G,$$
each induced character is the character of a real representation.  So, the irreducible character $\pm \chi^*$ is the integer linear combination of characters of real representations.  By Lemma \ref{linearcomblemma}, we have $\ep(\pm\chi^*) = 1$.
\end{proof}

Recall that the Gelfand-Graev character of $G$, which we will denote by $\Gamma$, is the character of the representation obtained by inducing a non-degenerate linear character from the unipotent subgroup of $G$ up to $G$ (see \cite[Section 8.1]{Ca85} for a full discussion).  A {\em regular} character of $G$ is defined as an irreducible character of $G$ which appears as a constituent of $\Gamma$.  It is well known that the Gelfand-Graev character has a multiplicity-free decomposition into irreducible characters of $G$.  We define the virtual character $\Xi$ of $G$ to be the dual of $\Gamma$, that is, $\Xi = \Gamma^*$.  A {\em semisimple} character of $G$ may be defined as an irreducible character $\chi$ of $G$ such that $\langle \Xi, \chi \rangle \neq 0$.  Equivalently, a semisimple character of $G$ may be defined to be an irreducible character with the property that its average value on regular unipotent elements is nonzero.  In the case that $p$ is a good prime for $G$ (see \cite[Section 1.14]{Ca85} for a definition), the semisimple characters of $G$ are exactly those for which the degree is not divisible by $p$.  

Alvis-Curtis duality maps a regular character of $G$ to plus or minus a semisimple character, and a semisimple character of $G$ to plus or minus a regular character (see \cite[Section 8.3]{Ca85}).  This gives us the following Corollary of Theorem \ref{FSdual}.

\begin{cor} The number of orthogonal (resp. symplectic) regular characters of $G = {\bf G}(\fifi)$ is equal to the number of orthogonal (resp. symplectic) semisimple characters of $G$.
\end{cor}

\section{Central characters} \label{CentralCharacters}

For an irreducible representation $\pi$, with character $\chi$, of a finite group, let $\omega_{\pi}$, or $\omega_{\chi}$, denote the central character of $\pi$ given by Schur's Lemma.  

As in the previous section, here we let $\bG$ be a connected reductive group over $\fifibar$ with connected center, which is defined over $\fifi$ with Frobenius map $F$.  In \cite[Theorem 3]{prasad}, D. Prasad obtains the following result connecting the Frobenius-Schur indicator of a regular character of $G$ to the value of the central character at a specific element.

\begin{thm}[Prasad] \label{prasaddual} Let ${\bf G}$ be a connected reductive group over $\bar{\mathbb{F}}_q$ which is defined over $\fifi$, with Frobenius map $F$, and such that the center $Z({\bf G})$ is connected.  Then there exists an element $z$ in the center of ${\bf G}^F = G$ such that for any real-valued regular character $\chi$ of $G$, we have $\ep(\chi) = \omega_{\chi}(z)$.
\end{thm}

We now prove that the central character $\omega_{\chi}$, like the Frobenius-Schur indicator, is preserved under duality. 

\begin{lemma} \label{centraldual} Let $\chi$ be an irreducible character of $G = {\bf G}^F$.  Then $\omega_{\chi} = \omega_{\pm \chi^*}$.
\end{lemma}
\begin{proof} First, $Z({\bf G})$ is contained in any Borel subgroup ${\bf B}$ of $G$, since ${\bf B}$ is a maximal closed connected solvable subgroup of ${\bf G}$.  So, $Z({\bf G})$ is contained in any parabolic subgroup ${\bf P}$ of ${\bf G}$.  It follows that $Z({\bf G})^F$ is contained in any of the parabolic subgroups $P_J$ of $G = {\bf G}^F$.  By \cite[Proposition 3.6.8]{Ca85}, $Z({\bf G})^F = Z({\bf G}^F)$, and so the center of $G = {\bf G}^F$ is contained in every parabolic subgroup of $G$.  Now, if $z \in Z(G)$, and $P_{J}$ is a parabolic subgroup of $G$, then it follows from applying Equation (\ref{truncform}) that
$$T_{P/N}(\chi)(z) = \omega_{\chi}(z) T_{P/N}(\chi)(1).$$
When evaluating $\chi^*(z)$, we may factor out $\omega_{\chi}(z)$ from each term in the sum, and obtain $\chi^*(z) = \omega_{\chi}(z) \chi^*(1)$.  It immediately follows that $\omega_{\pm \chi^*} = \omega_{\chi}$.
\end{proof}

Theorem \ref{FSdual} and Lemma \ref{centraldual} immediately imply that Prasad's Theorem \ref{prasaddual} may be extended to semisimple characters.  We may extend the result even further as follows.  Let $\bG$ is a connected reductive group over $\fifibar$ which is defined over $\fifi$, with corresponding Frobenius map $F$, so that $\bG$ is also defined over ${\mathbb F}_{q^r}$ for any positive integer $r$, with Frobenius map $F^r$  Then $\bG^{F^r} = \bG(\fifir)$ is the set of $\fifir$-points of $\bG$, and $\bG^{F} \subset \bG^{F^r}$.

In order to further extend Theorem \ref{prasaddual}, we briefly summarize the main idea of its proof as it appears in \cite{prasad}.  Let $\bB$ be a Borel subgroup of $\bG$ which is defined over $\fifi$ (and thus defined over $\fifir$ for any $r \geq 1$), let $\bT$ be a maximal torus of $\bG$ contained in $\bB$ which is defined over $\fifi$, and let $\bN$ be the unipotent radical of $\bB$.  Through \cite[Theorem 2]{prasad}, the proof of Theorem \ref{prasaddual} is reduced to proving the existence of an element $s \in \bT(\fifi)$ such that $s$ acts by $-1$ on all of the simple root spaces of $\bN$.  This is proven in \cite[Theorem 3]{prasad}, and the element $s$ has the property that $s^2 = z \in Z(\bG(\fifi))$.  We make note of the fact that this same $s$ which acts by $-1$ on all of the simple root spaces of $\bN$ satisfies $s \in \bT(\fifir)$, and $s^2 = z \in Z(\bG(\fifi)) \subset Z(\bG(\fifir))$.  Thus, the central element obtained when applying Theorem \ref{prasaddual} to $\bG(\fifi)$ is the same central element obtained when applying Theorem \ref{prasaddual} to $\bG({\mathbb F}_{q^r})$ for any positive integer $r$.  By applying this, along with Lemma \ref{centraldual}, we obtain the following generalization of Theorem \ref{prasaddual}.

\begin{thm} \label{FScentralExt} Let $\bG$ be a connected reductive group over $\fifibar$ with connected center, which is defined over $\fifi$ with Frobenius map $F$, and is thus defined over ${\mathbb F}_{q^r}$ with Frobenius map $F^r$ for any positive integer $r$.  Then there exists an element $z \in Z(\bG(\fifi))$ such that for any real-valued regular or semisimple character $\chi$ of $\bG({\mathbb F}_{q^r})$, we have $\ep(\chi) = \omega_{\chi}(z)$.
\end{thm}

\section{Applications} \label{Applications}

\subsection{Frobenius-invariant Deligne-Lusztig characters}
For a connected reductive group $\bG$ over $\fifibar$ with connected center, defined over $\fifi$ with Frobenius map $F$, consider an $F$-rational maximal torus $\bT$ of $\bG$, and let $\theta: \bT(\fifi) \rightarrow \cplx^{\times}$ be a character.  Let $R_{\bT, \theta}$ be the Deligne-Lusztig virtual character associated with the pair $(\bT, \theta)$, which is a virtual character of the group $\bG(\fifi)$, originally constructed in \cite{dellusz} (see also \cite{Ca85, dmbook}).  By \cite[Corollary 7.3.5]{Ca85}, for example, $\pm R_{\bT, \theta}$ is an irreducible character of $\bG(\fifi)$ if and only if $\theta$ is in general position, that is, if $\theta$ is not fixed by any non-identity element of the Weyl group $W(\bT)^F$.  Again, we use $\pm$ to mean the correct sign is chosen to have an irreducible character.  We need the following result.

\begin{lemma} \label{DLcentchar} Let $R_{\bT, \theta}$ be a Deligne-Lusztig virtual character of $\bG(\fifi)$.  Then for any central element $z \in Z(\bG(\fifi))$, we have $R_{\bT, \theta}(z) = \theta(z) R_{\bT, \theta}(1)$.
\end{lemma}
\begin{proof}  Since any central element is semisimple and is contained in every torus, then by applying the character value formula for Deligne-Lusztig characters on semisimple elements given in \cite[Proposition 7.5.3]{Ca85}, we have
$$ R_{\bT, \theta}(z) = \frac{\ep_{\bT} \ep_{\bG}}{|\bT^F| |\bG^F|_p} \sum_{g \in \bG^F} \theta(z) = \frac{\ep_{\bT} \ep_{\bG} |\bG^F|_{p'}}{|\bT^F|} \theta(z), $$
where $\ep_{\bG}$ and $\ep_{\bT}$ denote $-1$ to the relative ranks of $\bG$ and $\bT$, respectively (see \cite[p. 197]{Ca85} for a definition).  From the formula for $R_{\bT, \theta}(1)$ given in \cite[Theorem 7.5.1]{Ca85}, we obtain $R_{\bT, \theta}(z) = \theta(z) R_{\bT, \theta}(1)$.
\end{proof}

If $\bT$ is an $F$-rational maximal torus of $\bG$, we now consider a character $\Theta: \bT(\fifisq) \rightarrow \cplx^{\times}$, and the Deligne-Lusztig character $R_{\bT, \Theta}$ of $\bG(\fifisq)$.  A class function $\chi$ of $\bG(\fifisq)$ is {\em Frobenius-invariant} if ${^F \chi} = \chi$, that is, if $\chi({^F g}) = \chi(g)$ for every $g \in \bG(\fifisq)$.  Let $\Nm_{\fifisq/\fifi}$ (or simply $\Nm$ when the fields are clear from the context) denote the norm map of $\fifisq$ down to $\fifi$, so that $\Nm(x) = x \cdot {^F x}$ for any $x \in \fifisq$.  Then $\Nm$ may be extended to any abelian subgroup of $\bG(\fifisq)$.  By \cite[Lemma 6.9]{gyoja}, $\pm R_{\bT, \Theta}$ is an irreducible Frobenius-invariant Deligne-Lusztig character of $\bG(\fifisq)$, if and only if $\Theta = \theta \circ \Nm_{\fifisq/\fifi}$, for some character $\theta: \bT(\fifi) \rightarrow \cplx$ such that $\pm R_{\bT, \theta}$ is an irreducible Deligne-Lusztig character of $\bG(\fifi)$.  

If $R_{\bT, \theta}$ is any Deligne-Lusztig character of $\bG^F$, then its conjugate is given by $R_{\bT, \bar{\theta}}$, from \cite[Proposition 7.2.3]{Ca85}, for example.  If we assume that $\pm R_{\bT, \theta}$ is an irreducible Deligne-Lusztig character of $\bG^F$, then it follows from \cite[Theorems 7.3.4 and 7.3.8]{Ca85} that it is real-valued if and only if there is an element of the Weyl group $W(\bT)^F$ whose action maps $\theta$ to $\bar{\theta}$.  We now have the following result.

\begin{thm} \label{DLthm} Let $\bG$ be a connected reductive group over $\fifibar$ with connected center, which is defined over $\fifi$ with Frobenius map $F$, and so defined over $\fifisq$ with the map $F^2$.  Let $\pm R_{\bT, \theta}$ be an irreducible real-valued Deligne-Lusztig character of $\bG(\fifi)$.  Let $\Theta = \theta \circ \Nm: \bT(\fifisq) \rightarrow \cplx^{\times}$.  Then $\pm R_{\bT, \Theta}$ is an irreducible, real-valued, Frobenius-invariant, Deligne-Lusztig character of $\bG(\fifisq)$, such that $\ep(\pm R_{\bT, \Theta}) = 1$.
\end{thm}
\begin{proof} Let us denote $\chi = \pm R_{\bT, \Theta}$, and $\psi = \pm R_{\bT, \theta}$.  First, irreducible Deligne-Lusztig characters are both regular and semisimple, and so by Theorem \ref{FScentralExt}, there exists an element $z \in Z(\bG(\fifi))$ such that $\ep(\psi) = \omega_{\psi}(z)$.  By Lemma \ref{DLcentchar}, we have $\omega_{\psi}(z) = \theta(z)$, and so $\theta(z) = \pm 1$.  By the discussion in the paragraph preceding this Lemma, since $\Theta = \theta \circ \Nm$, then $\chi = \pm R_{\bT, \Theta}$ is a Frobenius-invariant irreducible character of $\bG(\fifisq)$.  Since $\pm R_{\bT, \theta}$ is assumed to be real-valued, then there is an element $w \in W(\bT)^F$ such that the action of $w$ takes $\theta$ to $\bar{\theta}$.  Since $w$ is $F$-invariant, then $w$ acts on $\Theta$ as follows, where $t \in \bT(\fifisq)$:
$$ \Theta(wtw^{-1}) = \theta(w \Nm(t) w^{-1}) = \bar{\theta}(\Nm(t)) = \bar{\Theta}(t).$$
Since the action of $w \in W(\bT)^F \subset W(\bT)^{F^2}$ takes $\Theta$ to $\bar{\Theta}$, then $\chi$ is real-valued.  So, by Theorem \ref{FScentralExt}, $\ep(\chi) = \omega_{\chi}(z)$, and by Lemma \ref{DLcentchar}, $\omega_{\chi}(z) = \Theta(z)$.  Now we have
$$ \ep(\chi) = \Theta(z) = \theta(\Nm(z)) = \theta(z \cdot {^F z}).$$
Since $z \in Z(\bG(\fifi))$, then $\Nm(z) = z^2$, and so we finally have $\ep(\chi) = \theta(z)^2 = 1$, since $\theta(z) = \pm 1$.
\end{proof}

\subsection{Finite unitary groups}

We now let $\bG = \GL(n, \fifibar)$.  Considering elements of $\bG$ as invertible $n$-by-$n$ matrices with entries from $\fifibar$, for $g = (a_{ij}) \in \bG$, define $\bar{g}$ by $\bar{g} = (a_{ij}^q)$, and let ${^t g} = (a_{ji})$ denote the transpose map.  We now define a Frobenius map $F$ on $\bG$ by
$$ F(g) = \left( \begin{array}{ccc}  &  & 1 \\  & \adots &   \\ 1 &  &  \end{array} \right) {^t \bar{g}}^{-1}\left( \begin{array}{ccc}  &  & 1 \\  & \adots &   \\ 1 &  &  \end{array} \right).$$
Then the group of $\fifi$-points of $\bG$ is defined to be the finite unitary group, which we denote by $\bG^F = \U(n,\fifi)$.  We comment that if we define the Frobenius map $\tilde{F}$ by $\tilde{F}(g) = {^t \bar{g}^{-1}}$, then in fact $\bG^{\tilde{F}} \cong \bG^{F}$, but we choose the Frobenius map $F$ so that, in particular, the standard Borel subgroup of $\bG$ is stable under $F$.

The center of $\U(n,\fifi)$ is the set of all scalar matrices of the form $\alpha I$ such that $\alpha \in \fifibar$ and  $\alpha^{q+1} = 1$, and so $\alpha \in \fifisq$.  As already noted, the standard Borel subgroup $\bB_0$, consisting of upper triangular matrices, is an $F$-rational subgroup of $\bG$, and contains the maximally split standard torus $\bT_0$, consisting of diagonal matrices, which is also $F$-rational.  It may be directly checked that if $\alpha \in \fifisq^{\times}$, then in the case that $n$ is even, the element 
$$ s = {\rm diag}(\alpha, \alpha^{-q}, \ldots, \alpha, \alpha^{-q})$$
satisfies $s \in \bT_0^{F}$.  Finally, we let $\bN_0$ denote the unipotent radical of the standard Borel subgroup $\bB_0$.  

In the following, part (1) follows directly from \cite[Theorem 7(ii)]{Ohm1}, but we also give a proof below.

\begin{thm} \label{FSunitary} Let $\chi$ be a real-valued semisimple or regular character of $\U(n, \fifi)$.  Then:
\begin{enumerate} 
\item If $n$ is odd or $q$ is even, then $\ep(\chi) = 1$.
\item If $n$ is even and $q \equiv 1(\text{mod } 4)$, then $\ep(\chi) = \omega_{\chi}(-I)$.
\item If $n$ is even and $q \equiv 3(\text{mod } 4)$, then $\ep(\chi) = \omega_{\chi}(\beta I)$, where $\beta = t^2$ and $t \in \fifisq$ is such that $t^{q+1} = -1$.
\end{enumerate}
\end{thm}
\begin{proof} For each part, we apply Theorem \ref{FScentralExt} along with the proof of Theorem \ref{prasaddual}, as discussed in the paragraph preceding Theorem \ref{FScentralExt}.  We must find an element $s \in \bT_0^F$ such that $s$ acts by $-1$ on each of the simple root spaces of $\bN_0$, and then we take $z = s^2$ as the central element in Theorem \ref{FScentralExt}.  In this case, this means we must find an $s$ such that, for every $h \in \bN_0$, if $h = (a_{ij})$ and $shs^{-1} = (b_{ij})$, then $b_{i, i+1} = -a_{i, i+1}$ for $1 \leq i \leq n-1$.

In part (1), first consider the case that $q$ is even.  Since we are in characteristic $2$, the elements $s$ and $z$ may be taken to be the identity, and hence $\ep(\chi)=1$.  Now, if $n$ is odd (and $q$ is odd or even), we let $s$ be the element
$$s = {\rm diag}(-1, 1, \ldots, 1, -1) \in \bT_0^F.$$
Then conjugation by $s$ on any element of $\bG$ sends each superdiagonal entry to its negative, as desired.  So, $s^2 = 1$, and $\ep(\chi) = 1$ for any regular or semisimple real-valued character $\chi$.

In case (2), when $n$ is even and $q \equiv 1(\text{mod } 4)$, if we let $\gamma \in \fifi$ be such that $\gamma^2 = -1$, then we let 
$$ s = {\rm diag}(\gamma, -\gamma, \ldots, \gamma, -\gamma).$$
Then $s \in \bT_0^{F}$, and $s$ has the desired property.  Since $s^2 = -I$, the result follows.

Finally, in case (3), when $n$ is even and $q \equiv 3(\text{mod } 4)$, first let $t \in \fifisq$ such that $t^{q+1} = -1$.  Such an element exists since the polynomial $x^{q+1} + 1$ is a factor of $x^{q^2 - 1} -1$, which splits completely over $\fifisq$.  Now, $t^{-q} = -t$, and so if 
$$ s = {\rm diag}(t, -t, \ldots, t, -t), $$
then $s \in \bT_0^{F}$, and $s$ again acts as $-1$ on each of the simple root spaces of $\bN_0$.  Letting $\beta = t^2$, the desired central element is $z = \beta I$.
\end{proof}

We note that in all of the examples given in \cite{Gow3, Gow4, prasad}, the central element which may be used to find the Frobenius-Schur indicator of a real-valued regular character (or in some cases, any character) is either $I$ or $-I$.  Case (3) of Theorem \ref{FSunitary} seems to be the first example given for which such a central element is not $\pm I$.

\end{document}